\documentclass[11pt,a4paper,oneside]{article}
\usepackage[english]{babel}
\usepackage[T1]{fontenc}
\usepackage{amsmath}
\usepackage{amsthm}
\usepackage{amsfonts}
\usepackage{amssymb}
\usepackage{lmodern} 
\usepackage{indentfirst}		
\usepackage{microtype} 			
\usepackage[numbers]{natbib}
\usepackage{hyphenat}
\usepackage{pgfplots}  	
\usepackage[latin1]{inputenc}
\usepackage{indentfirst}

\definecolor{blackgreen}{RGB}{0,80,0}

\usepackage[plainpages=false,pdfpagelabels,pdfusetitle,pdfdisplaydoctitle, pdfstartpage=1, pdfstartview={{XYZ null null 1.00}}]{hyperref}
\hypersetup{
	pdftitle={critical nom},    
	pdfauthor={Mykael Cardoso and Luiz Gustavo Farah},    
	pdfnewwindow=true,      
	colorlinks=true,      
	linkcolor=blackgreen,         
	citecolor=red}

\usepackage[top=2cm, bottom=2cm, left=2cm, right=2cm]{geometry}
\usepackage{mathtools}
\mathtoolsset{showonlyrefs}

\newcommand{\Real}{\mathbb R}
\newtheorem{thm}{Theorem}[section]
\newtheorem{prop}[thm]{Proposition}
\newtheorem{lemma}[thm]{Lemma}

\newtheorem{coro}[thm]{Corollary}
\numberwithin{equation}{section}

\title{Blow-up solutions of the intercritical inhomogeneous NLS equation: the non-radial case}
\author{Mykael Cardoso and Luiz Gustavo Farah} 
\date{} 

\begin{document}
\maketitle
	
\begin{abstract}\noindent
In this paper we consider the inhomogeneous nonlinear Schr\"odinger (INLS) equation
\begin{align}\label{inls}
i \partial_t u +\Delta u +|x|^{-b} |u|^{2\sigma}u = 0, \,\,\, x \in \mathbb{R}^N
\end{align}
with $N\geq 3$. We focus on the intercritical case, where the scaling invariant Sobolev index $s_c=\frac{N}{2}-\frac{2-b}{2\sigma}$ satisfies $0<s_c<1$. In a previous work, for radial initial data in $\dot H^{s_c}\cap \dot H^1$, we prove the existence of blow-up solutions and also a lower bound for the blow-up rate. Here we extend these results to the non-radial case. We also prove an upper bound for the blow-up rate and a concentration result for general finite time blow-up solutions in $H^1$.
\end{abstract}

\section{Introduction}
In this work we consider the initial value problem (IVP) for the inhomogeneous nonlinear Schr\"odinger (INLS) equation
\begin{equation}
\begin{cases}
i \partial_t u + \Delta u + |x|^{-b} |u|^{2 \sigma}u = 0, \,\,\, x \in \mathbb{R}^N, \,t>0,\\
u(0) = u_0,
\end{cases}
\label{PVI}
\end{equation}
for $N\geq 3$ and some $b, \sigma>0$. This model is a generalization of the classical nonlinear Schr\"odinger (NLS) equation and also has applications in laser beam propagation upon a nonlinear optical medium \cite{GILL} and \cite{LIU}.

We are mainly interested in the intercritical regime. To understand this terminology, we recall that if $u(x,t)$ solves \eqref{PVI} so does $u_{\lambda}(x,t)=\lambda^{\frac{2-b}{2\sigma}}u(\lambda x, \lambda^2 t)$ and also $\|u_{\lambda}(0)\|_{\dot{H}^{s_c}}=\|u_0\|_{\dot{H}^{s_c}}$
where $s_c=\frac{N}{2}-\frac{2-b}{2\sigma}$. The mass-supercritical and energy-subcritical regime is such that $0<s_c<1$ and we can reformulate this condition as
\begin{equation}\label{IR}
\frac{2-b}{N}<\sigma<\frac{2-b}{N-2}.
\end{equation}

Over the last few years, the INLS equation has been the subject of a great deal of mathematical research. This is part of a recently growing interest in the global dynamics of NLS type equations lacking the usual symmetries. In the present case, the translation invariance is not present and there is a space-dependent singular coefficient in the nonlinearity. Several results concerning well-posedness theory, existence and concentration of blow-up solutions,  stability of solitary waves and asymptotic behavior of global  solutions have been recently obtained for the INLS model \cite{AT21}-\cite{CFGM20}, \cite{CG16}-\cite{GENSTU}, \cite{CARLOS}, \cite{GM21}, \cite{KLS21}, \cite{MMZ19} and \cite{M21}. In particular, $H^1$-solutions have the mass and the energy conserved in time, more precisely, if $u(t)$ is a solution to \eqref{PVI} on some time interval $I\subset \mathbb{R}$, $0\in I$, then for any $t\in I$
\begin{equation}\label{Mass}
M[u(t)] =\int  |u(x,t)|^2\, dx =M[u_0]
\end{equation}
and
\begin{equation}\label{Energy}
E[u(t)] =\frac12 \int  |\nabla u(x,t)|^2\, dx - \frac{1}{2\sigma+2}\int |x|^{-b}|u(x,t)|^{2\sigma+2}\,dx=E[u_0]. 
\end{equation}

The main goal of this paper is to remove the radial assumption in our previous work \cite{CF20}. This is in the same vein as the recent papers by \citet{BL21} and \citet{GM21}, where the authors reported results for the INLS equation in the non-radial case that have remained out of reach so far for the NLS equation. Although the presence of the term $|x|^{-b}$ in the nonlinearity introduces several challenging technical difficulties in the study of well-posedness and scattering theories, its decay away from the origin provides a new tool in the study of the global dynamics of solutions that avoid the need for a radial assumption.

Our first result is related to finite time solutions in $H^1$. The existence of such solutions for the INLS equation \eqref{PVI} was obtained by the second author \cite{Farah} in the virial space $\Sigma:=\{f\in H^1;\,\,|x|f\in L^2 \}$. This result was extended by \citet{dinh2017blowup} in radial setting and more recently by Ardila and the first author in \cite{AC20} and \citet{BL21} for general initial data. In particular, the result in  \cite{BL21} says that if $\frac{2-b}{N}<\sigma<\min\left\{\frac{2-b}{N-2},\frac{2}{N}\right\}$, then blow-up occurs in finite time. Under the same restriction on $\sigma$, we prove an universal space-time upper bound on the blow-up rate.

\begin{thm}\label{thmblowdisper}
	Let $N\geq3$, $0<b<\min\{\frac{N}{2},2\}$ and $\frac{2-b}{N}<\sigma<\min\left\{\frac{2-b}{N-2},\frac{2}{N}\right\}$. Let $u_0\in H^1 $ and assume that the maximal time of existence $T^{\ast}>0$ for the corresponding solution $u\in C([0,T^{\ast}):H^1 )$ of \eqref{PVI} is finite. Define $\beta=\frac{2-\sigma N}{b}$, then there exists a universal constant $C=C(u_0, \sigma, b, N)>0$ such that the following space-time upper bound holds
	\begin{align}\label{blowupestdisp}
	\int_{t}^{T^{\ast}}(T^{\ast}-\tau)\|\nabla u(\tau)\|_{L^2}^2\,d\tau\leq C(T^{\ast}-t)^{\frac{2\beta}{1+\beta}},
	\end{align}
for $t$ close enough to $T^{\ast}$.
\end{thm}

This type of result was first proved by Merle, Rapha\"{e}l, and Szeftel \cite{MRS2014} for the NLS equation and then extended in \cite{CF20} for the INLS equation, both in the radial case. Here we show that the radial assumption can be removed in the INLS setting. The new ingredient is a localized virial type inequality satisfied by the non-radial solutions of the INLS equation (see Lemma \ref{lemintradial}).

As a consequence of \eqref{blowupestdisp}, we deduce an upper bound on the blow-up rate along a sequence of times. Indeed, there exists a sequence $\{t_n\}_{n=1}^{+\infty}\subset [0,T^{\ast})$ with $t_n\to T^{\ast}$ such that
	\begin{align}
	\|\nabla u(t_n)\|_{L^2}\leq \frac{C}{(T^{\ast}-t_n)^{\frac{1}{1+\beta}}},\,\,\,\, \mbox{ as }\,\,\,\, n\to+\infty.
	\end{align}

Finite time solutions in $H^1$ also enjoy other important properties. For instance, from the $H^1$ local Cauchy theory obtained by \citet{CARLOS}, if $T^{\ast}<\infty$, then
\begin{equation}\label{GradBU1}
\|\nabla u(t)\|_{L^2}\to \infty, \,\,\, \mbox{as} \,\,\, t \to T^{\ast}.
\end{equation}
Moreover, the recent work by \citet{AT21} proved that these solutions obey the lower bound 
$$
\left\|\nabla u(t)\right\|_{L^2}\geq \frac{c}{(T^{\ast}-t)^{\frac{1-s_c}{2}}}.
$$

On the other hand, recalling that ${H}^{1} \subset L^{\sigma_c}$ for $\sigma_c=\frac{2N\sigma}{2-b}$, in a joint work with Guzm\'an we obtain the following Gagliardo-Nirenberg type inequality  \cite{CFG20}
\begin{align}\label{GNine}
\int |x|^{-b} |u(x)|^{2\sigma+2}\,dx\leq \frac{\sigma+1}{\|V\|_{L^{\sigma_c}}^{2\sigma}}\|\nabla u\|_{L^2}^2\|u\|_{L^{\sigma_c}}^{2\sigma},
\end{align}
where $V$ is a solution to elliptic equation 
\begin{align}\label{elptcpc1}
\Delta V+|x|^{-b}|V|^{2\sigma}V-|V|^{\sigma_c-2}V=0
\end{align}
 with minimal $L^{\sigma_c}$-norm. Applying the last inequality to the energy conservation \eqref{Energy} we have
\begin{equation}
	E(u_0) \geq \frac{1}{2}\,\|\nabla u(t)\|_{L^2}^2\left(1-\frac{\|u(t)\|_{L^{\sigma_c}}^{2\sigma}}{\|V\|_{L^{\sigma_c}}^{2\sigma}}\right),
	\end{equation}
which implies that finite time solutions must satisfy $\sup_{t\in [0,T^*)}\|u(t)\|_{L^{\sigma_c}}\geq \|V\|_{L^{\sigma_c}}$. In the next result, we prove that the $L^{\sigma_c}$ norm of such solutions concentrates around the origin in the non-radial case.

\begin{thm}\label{Tlpc}
Let $\sigma_c=\frac{2N\sigma}{2-b}$ such that $\dot{H}^{s_c} \subset L^{\sigma_c}$. 	Assume $N\geq3$, $0<b<\min\{\frac{N}{2},2\}$ and $\frac{2-b}{N}<\sigma<\min\left\{\frac{2-b}{N-2},\frac{2}{N}\right\}$. Given $u_0\in H^1 $ so that the maximal time of existence $T^{\ast}>0$ for the corresponding solution $u\in C([0,T^{\ast}):H^1 )$ of \eqref{PVI} is finite. Then there exist positive constants $c_0$ and $c_1$ depending only on $N,\sigma$ and $b$ such that 
\begin{equation}\label{fint}
\liminf_{t\to T^*}\int_{|x|\leq c_{u_0}\|\nabla u(t)\|^{-\frac{2-\sigma N}{b}}_{L^2}}|u(x,t)|^{\sigma_c}\,dx\geq c_0,
\end{equation}
where $c_{u_0}=c_1\max\left\{\|u_0\|_{L^2}, \|u_0\|_{L^2}^{\frac{2\sigma+2-\sigma N}{b}}\right\}$.
\end{thm}

Note that the concentration window size $c_{u_0}\|\nabla u(t)\|^{-\frac{2-\sigma N}{b}}_{L^2}$ can be made arbitrarily small for times close to the finite maximal time of existence in view of \eqref{GradBU1} and \eqref{Mass}. To prove this result we use a spatial and frequency localization also employed by \citet{HR07} for the radial 3D cubic NLS equation (see also the work \cite{CC18}, where Campos and the first author treat the radial INLS setting).

Next we turn to the IVP \eqref{PVI} with initial data $u_0\in \dot H^{s_c}\cap \dot H^1$. Recently, in collaboration with Guzm\'an \cite{CFG20}, we proved that the IVP \eqref{PVI} is locally well-posed in this space. In \cite{CF20}, applying the techniques developed by \citet{MR_Bsc} for the mass-supercritical NLS equation, we obtain the existence of radially symmetric solution non-positive energy solutions with finite maximal time of existence and the behavior of the $\dot H^{s_c}$ norm. Here we consider a general initial data. Our main results in this direction are the following.

\begin{thm}\label{scteo1} Let $N\geq3$, $0<b<\min\{\frac{N}{2},2\}$ and $\frac{2-b}{N}<\sigma<\min\left\{\frac{2-b}{N-2},\frac{2}{N}\right\}$. If $u_0\in \dot H^{s_c} \cap\dot H^1 $ and $E(u_0)\leq 0,$ then the maximal time of existence $T^{\ast}>0$ of the corresponding solution $u(t)$ to \eqref{PVI} is finite.
\end{thm}

\begin{thm}\label{scteo2}
	Let $\sigma_c=\frac{2N\sigma}{2-b}$ such that $\dot{H}^{s_c} \subset L^{\sigma_c}$. Assume $N\geq3$, $0<b<\min\{\frac{N}{2},2\}$ and $\frac{2-b}{N}<\sigma<\min\left\{\frac{2-b}{N-2},\frac{2}{N}\right\}$. Given $u_0\in \dot H^{s_c}\cap \dot H^1$ so that the maximal time of existence $T^{\ast}>0$ of the corresponding solution $u$ to \eqref{PVI} is finite and satisfies 
\begin{align}\label{esttl}
\left\|\nabla u(t)\right\|_{L^2}\geq \frac{c}{(T^{\ast}-t)^{\frac{1-s_c}{2}}},
\end{align}
for some constant $c=c(N,\sigma)$ and $t$ close enough to $T^{\ast}$. Then there exists $\gamma=\gamma(N,\sigma,b)>0$ such that
	\begin{align}\label{rate}
	c\|u(t)\|_{\dot H^{s_c}}\geq \|u(t)\|_{L^{\sigma_c}}\geq |\log (T^{\ast}-t)|^{\gamma},\,\,\,\,as\,\,t\to T^{\ast}.
	\end{align}
\end{thm} 

The condition \eqref{esttl} is very natural and it is easily deduced if a local Cauchy theory in $\dot{H}^{1}$ is available (see \citet{CaWe89} and also the Introduction in \citet{MR_Bsc} for the argument in the NLS case). Moreover, if we additionally assume that $u_0\in H^1\subset \dot H^{s_c}\cap \dot H^1$, it is automatically satisfied as we mentioned before.

As a consequence of the previous theorem and the well-posedness theory in $\dot H^{s_c}\cap \dot H^1$ obtained in \cite{CFG20} we deduce that the $\dot{H}^{s_c}$-norm of the blow-up solution cannot be uniformly bounded.
\begin{coro}\label{cor13}
Assume $N\geq3$, $0<b<\min\{\frac{N}{2},2\}$ and $\frac{2-b}{N}<\sigma<\min\left\{\frac{2-b}{N-2},\frac{2}{N}\right\}$. Given $u_0\in \dot H^{s_c}\cap \dot H^1$ so that the maximal time of existence $T^{\ast}>0$ of the corresponding solution $u$ to \eqref{PVI} is finite, then
\begin{equation}\label{Bup}
\limsup_{t\rightarrow T^{\ast}}\|u(t)\|_{\dot H^{s_c}}=+\infty.
\end{equation}
\end{coro}
Note that the previous result does not hold in the mass-critical case ($s_c = 0$), since the scaling invariant $L^2$-norm is preserved by the mass conservation \eqref{Mass}. 

It should be emphasized that all the above results are still unknown for the classical NLS equation in the non-radial case. In our proofs we need the restriction $\frac{2-b}{N}<\sigma<\min\left\{\frac{2-b}{N-2},\frac{2}{N}\right\}$, which implies that the argument cannot be extended to the NLS setting since $b=0$ in this case. Moreover, if $b>4/N$, then $\frac{2-b}{N-2}<\frac{2}{N}$ and the result covers all the intercritical range \eqref{IR}. Finally, the assumption $N\geq 3$ can probably be relaxed and we assume this condition due to the local theory in $\dot H^{s_c}\cap \dot H^1$ obtained in \cite{CFG20}.

This paper is organized as follows. In Section \ref{sec3}, we prove a virial type estimate and use it to prove Theorem \ref{thmblowdisper}. The proof of the concentration result Theorem \ref{Tlpc} is also presented in this section. A non-radial interpolation estimate based on a Morrey-Campanato type semi-norm is presented in Section \ref{sec4} (see Lemma \ref{lemaradialGN}). The non-radial interpolation estimate and the virial type estimate are the main tools to obtain the existence of blow-up solutions and also a lower bound for the blow-up rate.

\section{Blow-up solutions in $H^1$}\label{sec3}

\subsection{A virial type estimate}

Let $\phi$ be a non-negative radial function $\phi\in C^{\infty}_0(\Real^N)$, such that
\begin{align}\label{phi}
	\phi(x)=
	\left\{
	\begin{array}{ll}
		|x|^2,&\mbox{ se }|x|\leq 2\\
		0,&\mbox{ se }|x|\geq 4
	\end{array}
	\right.
\end{align}
satisfying
\begin{align}\label{nablaphi}
	\phi(x)\leq c|x|^2, \quad |\nabla \phi (x)|^2\leq c\phi(x) \quad \mbox{and} \quad \partial_r^2\phi(x)\leq 2, \quad \mbox{for all } x\in \Real^N,
\end{align}
with $r=|x|$. Then, define $\phi_R(x)=R^2\phi\left(\frac{x}{R}\right)$. 
Consider $u\in C([0,\tau_*],\dot H^{s_c}\cap \dot{H}^1)$ a solution to \eqref{PVI}. For any $R>0$ and $t\in [0,\tau_*]$ define the function
\begin{align}\label{virial}
	z_R(t)=\displaystyle\int\phi_R|u(t)|^2\,dx.
\end{align} 

From direct computations (see, for instance, Proposition 7.2 in \cite{FG20}), we obtain
\begin{equation}\label{zR'2}
	z_R'(t)=2\mbox{Im}\int \nabla\phi_R\cdot\nabla u(t)\overline{u}(t)\,dx
\end{equation}
and 
\begin{align}\label{zR''}
	z_R''(t)=&4\mbox{Re} \sum_{j,k=1}^{N}\int \partial_ju(t)\,\partial_k\overline u(t)\,\partial^2_{jk}\phi_R\,dx-\int |u(t)|^2 \Delta^2\phi_R\nonumber\\
	&-\frac{2\sigma}{\sigma+1}\int|x|^{-b}|u(t)|^{2\sigma+2}\Delta\phi_R\,dx+\frac{2}{\sigma+1}\int\nabla\left(|x|^{-b}\right)\cdot \nabla\phi_R|u(t)|^{2\sigma+2}\,dx.
\end{align}

From the Cauchy-Schwarz inequality and mass conservation \eqref{Mass}, we have
\begin{equation}\label{zR-bound}
z_R(t)\leq cR^2\|u_0\|^2_{L^2}
\end{equation}
and
\begin{equation}\label{zR'-bound}
z'_R(t)\leq cR\|\nabla u(t)\|_{L^2}\|u_0\|_{L^2}.
\end{equation}

On the other hand, since $\phi$ is radial, we have
\begin{align}\label{virial1}
	z'(t)=2\,\mbox{Im}\int \partial_r\phi_R\frac{x\cdot \nabla u(t)}{r}\overline{u}(t)\,dx
\end{align}
and 
\begin{align}\label{virial2}
	z_R ''(t)=&4\int \frac{\partial_r\phi_R}{r}|\nabla u(t)|^2\,dx+4\int \left(\frac{\partial_r^2\phi_R}{r^{2}}-\frac{\partial_r \phi_R}{r^3}\right)|x\cdot \nabla u(t)|^2\,dx-\int|u(t)|^2 \Delta^2\phi_R
	\,dx \nonumber\\
	&-\frac{2\sigma}{\sigma+1}\int \left[\partial^2_r\phi_R +\left(N-1+\frac{b}{\sigma}\right)\frac{\partial_r \phi_R}{r}\right]|x|^{-b}|u(t)|^{2\sigma+2}\,dx,
\end{align}
where $\partial_r$ denotes the derivative with respect to $r=|x|$.

Consider the following functional
\begin{align}
	P[u]&=\int |\nabla u|^{2}\,dx-\frac{N\sigma+b}{2\sigma+2}\int |x|^{-b}|u|^{p+1}\,dx,\label{P}
\end{align}
then can rewrite $z_R''(t)$ in \eqref{virial2} as
\begin{align}
	z_R''(t)=8P[u(t)]+K_1+K_2+K_3,
\end{align}
where
\begin{align}
	K_1=&4\int \left(\frac{\partial_r\phi_R}{r}-2\right)|\nabla u(t)|^2\,dx+4\int \left(\frac{\partial^2_r\phi_R}{r^2}-\frac{\partial_r \phi_R}{r^3}\right)|x\cdot \nabla u(t)|^2\,dx,\\
	K_2=&-\frac{2\sigma}{\sigma+1}\int\left[\partial^2_r\phi_R+\left(N+1-\frac{b}{\sigma}\right)\frac{\partial_r \phi_R}{r}-2N-\frac{2b}{\sigma}\right]|x|^{-b}|u(t)|^{2\sigma+2}\,dx,\\
	K_3=&-\int|u(t)|^2\Delta^2\phi_R\,dx.
\end{align}
We claim that there exists $c>0$ such that
\begin{align}
	z_R''(t)\leq& 8P[u(t)]+c\left(\frac{1}{R^2}\int_{2R\leq |x|\leq 4R}|u(t)|^2\,dx+\int_{|x|\geq R}|x|^{-b}|u(t)|^{2\sigma+2}\,dx\right)\label{zR8Q}.
\end{align}
Indeed, we first show that $K_1\leq 0$. To this end, we consider the following region in $\mathbb R^N$
\begin{align}
	\Omega=\left\{x\in \mathbb R^N;\,\,\frac{\partial^2_r\phi_R(x)}{r^2}-\frac{\partial_r\phi_R(x)}{r^3}\leq 0\right\}.
\end{align}
Since $\partial^2_r\phi_R\leq 2$ it follows that $\partial_r\phi_R(|x|)\leq 2|x|$, for all $x\in \mathbb R^N$. Now, splitting the integration and using the Cauchy-Schwartz inequality, we get
\begin{align}
	K_1=&\,\,4\int\left(\frac{\partial_r\phi_R}{r}-2\right)|\nabla u(t)|^2\,dx+4\int\left(\frac{\partial_r^2\phi_R}{r^2}-\frac{\partial_r\phi_R}{r^3}\right) |x\cdot \nabla u(t)|^2\,dx\nonumber\\
	=&\,\,4\int_{\Omega} \left(\frac{\partial_r\phi_R}{r}-2\right)|\nabla u(t)|^2\,dx+4\int_{\Omega} \left(\frac{\partial_r^2\phi_R}{r^2}-\frac{\partial_r\phi_R}{r^3}\right) |x\cdot \nabla u(t)|^2\,dx\nonumber\\
	&+4\int_{\mathbb R^N\backslash\Omega} \left(\frac{\partial_r\phi_R}{r}-2\right)|\nabla u(t)|^2\,dx+4\int_{\mathbb R^N\backslash\Omega} \left(\frac{\partial_r^2\phi_R}{r^2}-\frac{\partial_r\phi_R}{r^3}\right) |x\cdot \nabla u(t)|^2\,dx\nonumber\\
	\leq&\,\, 4\int_{\mathbb R^N\backslash\Omega} \left(\frac{\partial_r\phi_R}{r}-2\right)|\nabla u(t)|^2\,dx+4\int_{\mathbb R^N\backslash\Omega} \left(\partial_r^2\phi_R-\frac{\partial_r\phi_R}{r}\right)\frac{|x|^2}{r^2}|\nabla u(t)|^2\,dx\nonumber\\
	=&\,\, 4\int \left(\partial_r^2\phi_R-2\right)|\nabla u(t)|^2\,dx\leq 0,\label{K_1}
\end{align}
where in the last inequality we have used the assumption \eqref{nablaphi}.

To estimate $K_2$, first note that if $0\leq |x|\leq 2R$, then
\begin{align}
	\partial_r\phi_R(x)=2|x|=2r,\,\,\,\,\,\,\partial^2_r\phi_R(x)=2,
\end{align}
which implies
\begin{align}
	\partial^2_r\phi_R+\left(N-1+\frac{b}{\sigma}\right)\frac{\partial_r\phi_R}{r}-2N-\frac{2b}{\sigma}=0, \quad \mbox{for} \quad 0\leq |x|\leq 2R.
\end{align}
Hence,
\begin{align}
	\mbox{supp} \left[\partial^2_r\phi_R+\left(N-1+\frac{b}{\sigma}\right)\frac{\partial_r\phi_R}{r}-2N-\frac{2b}{\sigma}\right]\subset (2R,\infty),
\end{align}
and thus, there exists $c>0$ such that
\begin{align}
	K_2\leq& c\int_{|x|\geq R}|x|^{-b}|u(t)|^{2\sigma+2}\,dx\label{K_2}.
\end{align}
Moreover, again by definition of $\phi_R$, we also conclude that 
\begin{align}
	K_3\leq \frac{c}{R^{2}}\int_{2R\leq |x|\leq 4R} |u(t)|^{2}\,dx\label{K_3}.
\end{align}
Collecting \eqref{K_1}, \eqref{K_2} and \eqref{K_3} we deduce \eqref{zR8Q}.

Next, from the energy conservation \eqref{Energy}, we can write
\begin{align}
	P[u(t)]= -\sigma s_c\|\nabla u(t)\|_{L^2}^2+2(\sigma s_c+1)E[u_0].
\end{align}
Then, combining the last identity with the inequality \eqref{zR8Q}, we deduce the following lemma.
\begin{lemma}\label{lemintradial}
	Let $N\geq3$, $0<b<\min\{\frac{N}{2},2\}$  and $\frac{2-b}{N}<\sigma<\min\left\{\frac{2-b}{N-2},\frac{2}{N}\right\}$. If $u\in C([0,\tau_*]: \dot H^{s_c}\cap \dot H^1 )$ is a solution to \eqref{PVI} with initial data $u(0)=u_0$, then there exists $c>0$ depending only on $N, \sigma, b$ such that for all $R>0$ and $t\in[0, \tau_\ast]$ we have
	\begin{align}
	8\sigma s_c\|\nabla u(t)\|^2_{L^2}+z''_R(t)&-16(\sigma s_c+1)E[u_0]\leq c\left(\frac{1}{R^2}\int_{2R\leq |x|\leq 4R}|u(t)|^2\,dx+\int_{|x|\geq R}|x|^{-b}|u(t)|^{2\sigma+2}\,dx\right).\label{estintradial}
	\end{align}
\end{lemma}

\subsection{Upper bound for the blow-up rate}

As an application of Lemma \ref{lemintradial} we obtain our first main result.

\begin{proof}[Proof of Theorem \ref{thmblowdisper}] Let $R,\varepsilon>0$ real numbers to be chosen later. First, using  interpolation and Sobolev embedding, or just the Gagliardo-Nirenberg inequality (see, for instance, Weinstein \cite[inequality $(I.2)$]{W83}), we have\footnote{This is the step where we explore the decaying factor in the nonlinearity instead of the radial assumption employed by Merle, Rapha\"{e}l and Szeftel \cite[Theorem 1.1]{MRS2014}.}
\begin{align}\label{RegGN}
	\int_{|x|\geq R}|x|^{-b}|u(t)|^{2\sigma+2}\,dx&\leq c\frac{1}{R^b}\|u(t)\|_{L^{\frac{2N}{N-2}}}^{\sigma N}\|u(t)\|_{L^2}^{2\sigma+2-\sigma N}\\
	&\leq c\frac{1}{R^b}\|\nabla u(t)\|_{L^2}^{\sigma N}\|u(t)\|_{L^2}^{2\sigma+2-\sigma N}.
\end{align}
Then, given $\varepsilon>0$, by the Young inequality and mass conservation \eqref{Mass} we conclude, for some constant $C_{\varepsilon}>0$, that
	\begin{align}
	\int_{|x|\geq R}|x|^{-b}|u(t)|^{2\sigma+2}\,dx&\leq \frac{c}{R^{b}}\|\nabla u\|_{L^{2}}^{\sigma N}\|u_0\|^{2\sigma+2-\sigma N}_{L^2} \leq \varepsilon \|\nabla u\|_{L^2}^2+C_{\varepsilon}\frac{\|u_0\|_{L^2}^{\frac{2(2\sigma+2-\sigma N)}{2-\sigma N}}}{R^{\frac{2b}{2-\sigma N}}},\label{case2}
	\end{align}
where we have used that $\sigma <2/N$.

Combining the inequality \eqref{estintradial}, energy and mass conservation \eqref{Energy}-\eqref{Mass} and the inequality \eqref{case2} we deduce that
	\begin{align}
	8\sigma s_c\int |\nabla u(t)|^2
	\,dx+z''_R(t)&\leq C_{u_0}\left(
	1+\frac{1}{R^2}+\int_{|x|\geq R}|x|^{-b}|u(t)|^{2\sigma+2}\,dx\right)\nonumber
	\\
	&\leq C_{u_0}\left(1+\frac{1}{R^2}+\varepsilon\int|\nabla u(t)|^2\,dx+\frac{C_{\varepsilon}}{R^{\frac{2b}{2-\sigma N}}}\right).
	\end{align}
Fix $\varepsilon>0$ small enough such that $8\sigma s_c-\varepsilon C_{u_0}>\sigma s_c$. Therefore, there exists a universal constant $C=C(u_0, \sigma, b, N)>0$ such that
\begin{align}
	\sigma s_c\int|\nabla u(t)|^2\,dx+z''_R(t)\leq C\left(1+\frac{1}{R^2}+\frac{1}{R^{\frac{2b}{2-\sigma N}}}\right).
	\end{align}
Now, since $\frac{2b}{2-\sigma N}>2$ (or $\sigma>\frac{2-b}{N}$) , if $R\ll1$, then
	\begin{align}\label{ctR}
	\sigma s_c\int|\nabla u(t)|^2\,dx+z''_R(t)\leq \frac{C}{R^{\frac{2b}{2-\sigma N}}}.
	\end{align}
	
The rest of the argument is as in the proof of Theorem 1.1 in Merle, Rapha\"{e}l and Szeftel \cite{MRS2014} (see also Theorem 1.4. in \cite{CF20}).

\end{proof}

\subsection{$L^{\sigma_c}$-norm concentration}
In this subsection we prove Theorem \ref{Tlpc}. We start by introducing some notation. Let $\phi \in C^{\infty}_0(\Real^N)$ be a positive radial cut-off solution such that
$$
\phi(x)=
\left\{
\begin{array}{ll}
	1&,\,\mbox{ if } |x|\leq 1,\\
	0&,\,\mbox{ if } |x|\geq 2.
\end{array}
\right.
$$
For $R(t)>0$, define the inner and outer spatial localizations of $u(x,t)$ at radius $R(t)$ as
$$u_1(x,t)=\phi\left(\frac{x}{R(t)}\right)u(x,t)\,\,\,\,\mbox{ and }\,\,\,\,u_2(x,t)=\left(1-\phi\left(\frac{x}{R(t)}\right)\right)u(x,t
).$$
Moreover, let $\chi \in C^{\infty}_0(\Real^N)$ be a radial function such that $\chi(x)=0$ for $|x|\geq 1$ and $\widehat{\chi}(0)=1$. For $\rho(t)$, define the inner and outer frequency localizations at radius $\rho(t)$ of $u_1(x,t)$ as $$\widehat{u}_{1L}(\xi,t)=\widehat{\chi}\left(\frac{\xi}{\rho(t)}\right)\widehat{u}_1(\xi,t)\,\,\,\,\mbox{ and }\,\,\,\,\widehat{u}_{1H}(\xi,t)=\left(1-\widehat{\chi}\left(\frac{\xi}{\rho(t)}\right)\right)\widehat{u}_1(\xi,t).$$
It is clear that
\begin{equation}\label{u1H4}
\|\nabla {u}_{1L}(t)\|_{L^2}\leq c\|\nabla {u}_{1}(t)\|_{L^2} \quad \mbox{ and }\quad \|\nabla {u}_{1H}(t)\|_{L^2}\leq c\|\nabla {u}_{1}(t)\|_{L^2}.
\end{equation}
%
\begin{proof}[Proof of Theorem \ref{Tlpc}] 
First, from \eqref{GradBU1} and the energy conservation \eqref{Energy} we deduce
$$
\lim_{t\to T^{\ast}}\frac{\int |x|^{-b}|u(t)|^{2\sigma+2}\,dx}{\|\nabla u(t)\|^2_{L^2}}=\sigma + 1.
$$ 
Hence, for $t$ close to $T^{\ast}$, we have
 \begin{align}\label{equ}
		\|\nabla u(t)\|_{L^2}^{2}&\leq  \int |x|^{-b}|u(t)|^{2\sigma+2}\,dx\nonumber\\
		&\leq \int |x|^{-b}|u_{1L}(t)|^{2\sigma+2}\,dx+\int |x|^{-b}|u_{1H}(t)|^{2\sigma+2}\,dx+\int |x|^{-b}|u_{2}(t)|^{2\sigma+2}\,dx,
	\end{align}
where in the last line we have used that $u=u_{1L}+u_{1H}+u_2$ with pairwise disjoint supports.
	
For a constant $a_1>0$, define
\begin{equation}\label{R(t)}
R(t)=a_1\frac{\max\left\{\|u_0\|_{L^2}, \|u_0\|_{L^2}^{\frac{2\sigma+2-\sigma N}{b}}\right\}}{\|\nabla u(t)\|^{\frac{2-\sigma N}{b}}_{L^2}}.
\end{equation}
Applying \eqref{RegGN}, we obtain\footnote{In \citet[Theorem 1.2]{HR07}, the authors used in this part the radial Gagliardo-Nirenberg estimate 
$$
\|f\|^4_{L^4(|x|\geq R)}\leq \frac{c}{R^2}\|\nabla f\|_{L^2(|x|\geq R)}\|f\|^3_{L^2(|x|\geq R)},
$$
 and hence, they need the radial restriction. Here, we use the decay of $|x|^{-b}$ away from the origin to obtain the desired estimate in the general case.}
	\begin{align}\label{eq_48}
		\int |x|^{-b}|u_{2}(t)|^{2\sigma+2}\,dx&=\int |x|^{-b}\left|\left(1-\phi\left(\frac{x}{R(t)}\right)\right)u(t
)\right|^{2\sigma+2}\,dx\nonumber\\&\leq c \int_{|x|\geq R(t)}|x|^{-b}|u(t)|^{2\sigma+2}\,dx \leq \frac{c}{R(t)^b}\|\nabla u(t)\|_{L^2}^{\sigma N}\|u_0\|_{L^2}^{2\sigma +2-\sigma N}\\
&\leq \frac{c}{a^b_1} \|\nabla u(t)\|_{L^2}^2
		\leq \frac{1}{4}\|\nabla u(t)\|_{L^2}^2,
	\end{align}
where we have chosen $a_1>0$ sufficiently large such that the last inequality holds.

Now, from the Gagliardo-Nirenberg inequality \eqref{GNine} and the Sobolev embedding $\dot{H}^{s_c} \subset L^{\sigma_c}$, we get
	\begin{align}\label{u1H1}
		\int |x|^{-b}|u_{1H}(t)|^{2\sigma+2}\,dx&\leq c \|\nabla u_{1H}(t)\|_{L^2}^2\|u_{1H}(t)\|_{L^{\sigma_c}}^{2\sigma}\leq c\|\nabla u_{1H}(t)\|_{L^2}^2\|u_{1H}(t)\|_{\dot{H}^{s_c}}^{2\sigma}\nonumber\\
		&=c\|\nabla u_{1H}(t)\|_{L^2}^2\left\||\xi|^{s_c}\left(1-\widehat{\chi}\left(\frac{\xi}{\rho(t)}\right)\right)\widehat{u}_1(t)\right\|_{L^2}^{2\sigma}
	\end{align}
We claim that
\begin{equation}\label{u1H2}
|\xi|^{s_c}\left(1-\widehat{\chi}\left(\frac{\xi}{\rho(t)}\right)\right)\leq c \frac{|\xi|}{\rho(t)^{1-s_c}}.
\end{equation}
Indeed, by the mean value theorem
	$$|1-\widehat{\chi}(\xi)|=|\widehat{\chi}(0)-\widehat{\chi}(\xi)|\leq c\min\{1,|\xi|\}.$$
Therefore, for $|\xi|\leq \rho(t)$, we have
	$$|\xi|^{s_c}\left|1-\widehat{\chi}\left(\frac{\xi}{\rho(t)}\right)\right|\leq c|\xi|^{s_c}\frac{|\xi|}{\rho(t)}\leq c\frac{|\xi|}{\rho(t)^{1-s_c}}$$
	and if $|\xi|\geq \rho(t)$, then
	$$|\xi|^{s_c}\left|1-\widehat{\chi}\left(\frac{\xi}{\rho(t)}\right)\right|\leq c|\xi|^{s_c}= |\xi|^{s_c}\frac{|\xi|}{|\xi|}\leq \frac{|\xi|}{\rho(t)^{1-s_c}}.$$
	Moreover, since $\sigma>\frac{2-b}{N}$ and $\|\nabla u(t)\|_{L^2}>1$ for $t$ close enough to $T^{\ast}$, from the definition of $R(t)$ \eqref{R(t)} we get
	\begin{align}
		\left\|\nabla u_1(t)\right\|_{L^2} &\leq \frac{c}{R(t)}\|u_0\|_{L^2}+c\|\nabla u(t)\|_{L^2}\\
		&\leq  c\,\left(\|\nabla u(t)\|^{\frac{2-\sigma N}{b}}_{L^2}+\|\nabla u(t)\|_{L^2}\right)\\ 
		&\leq c\|\nabla u(t)\|_{L^2}. \label{u1H3}
	\end{align}
For a constant $a_2>0$, let
\begin{equation}\label{rho(t)}
\rho(t)=a_2\|\nabla u(t)\|_{L^2}^{\frac{1}{1-s_c}}.
\end{equation}
Collecting the estimates \eqref{u1H4}, \eqref{u1H1}, \eqref{u1H2}, \eqref{u1H3} and choosing $a_2>0$ sufficiently large we have 
\begin{align}\label{eq_410}
\int |x|^{-b}|u_{1H}(t)|^{2\sigma+2}\,dx\leq c\|\nabla u_1(t)\|_{L^2}^2\left\|\frac{|\xi|}{\rho(t)^{1-s_c}}\widehat{u}_{1}(t)\right\|_{L^2}^{2\sigma}\leq c\frac{\|\nabla u(t)\|_{L^2}^{2\sigma+2}}{\rho(t)^{2\sigma(1-s_c)}}\leq \frac{1}{4}\|\nabla u(t)\|_{L^2}^2.
	\end{align}
	Therefore, in view of the inequality \eqref{equ}, \eqref{eq_48} and \eqref{eq_410}, we deduce
	\begin{equation}\label{inq}
		\left\|\nabla u(t)\right\|_{L^{2}}^{2}\leq c\int |x|^{-b}|u_{1L}(t)|^{2\sigma+2}\,dx.
	\end{equation}
Next, using again the Gagliardo-Nirenberg inequality \eqref{GNine} and the estimates \eqref{u1H4}, \eqref{u1H3} and \eqref{inq}, we obtain
	\begin{align*}
		\int |x|^{-b}|u_{1L}(t)|^{2\sigma+2}\,dx&\leq c\|\nabla u_{1L}(t)\|_{L^2}^{2}\|u_{1L}(t)\|_{L^{\sigma_c}}^{2\sigma}\\
		&\leq c\|u_{1L}(t)\|_{L^{\sigma_c}}^{2\sigma}\int |x|^{-b}|u_{1L}(t)|^{2\sigma+2}\,dx.
	\end{align*}
Moreover, Young's convolution inequality yields
	$$\|u_{1L}(t)\|_{L^{\sigma_c}}=\|\rho^N\chi(\rho\,\cdot)\ast u_1(t)\|_{L^{\sigma_c}}\leq c\|u_1(t)\|_{L^{\sigma_c}}.$$
From the last two inequalities it follows that $\|u_1(t)\|_{L^{\sigma_c}}$ is bounded from below by a universal constant $c>0$ independent of the initial data $u_0$ and time $t$. Finally, by definition of $u_1(t)$, we get
$$
c^{\sigma_c}\leq \|u_1(t)\|_{L^{\sigma_c}}^{\sigma_c}=\left\|\phi\left(\frac{x}{R(t)}\right)u(t)\right\|_{L^{\sigma_c}}^{\sigma_c}\leq \int_{|x|\leq2R(t)}|u(t)|^{\sigma_c}\,dx=\|u(t)\|^{\sigma_c}_{L^{\sigma_c}_{\{|x|\leq 2R(t)\}} },
$$
which, by the definition of $R(t)$ \eqref{R(t)}, implies \eqref{fint} and complete the proof of Theorem \ref{Tlpc}.

%
%
\end{proof}

As a consequence of the above argument, if we further assume that $\|u_1(t)\|_{L^{\sigma_c}}$ is not bounded above, then there exist a sequence $t_n\to T^{\ast}$ such that
\begin{equation}\label{inft}
\int_{|x|\leq c_{u_0}\|\nabla u(t_n)\|^{-\frac{2-\sigma N}{b}}_{L^2}}|u(x,t_n)|^{\sigma_c}\,dx\to +\infty.
\end{equation}
On the other hand, if $\|u_1(t)\|_{L^{\sigma_c}}$ is bounded, then it may be possible to prove that the concentration window shrinks at a different rate following the strategy of \citet[Theorem 1.2]{HR07}. We have decided not to explore this for the moment, since our main purpose here is to show how to remove the radial assumption and still obtain a concentration result around the origin for the INLS equation.

\section{Blow-up solutions in $\dot H^{s_c}\cap \dot H^1$}\label{sec4}
\subsection{Non-radial Gagliardo-Nirenberg inequality}
 
We first recall the following scaling invariant Morrey-Campanato type semi-norm used in \citet{MR_Bsc} (see also \cite{CF20})
\begin{equation}\label{defrho}
\rho(u,R)=\sup_{R'\geq R}\frac{1}{(R')^{2s_c}}\int_{R'\leq |x|\leq 2R'}|u|^2\,dx.
\end{equation}
It is easy to see that $\rho(u,R)$ is non-increasing in $R>0$. Moreover, by Holder's inequality, there exists a universal constant $c>0$ such that for all $u\in L^{\sigma_c}$ and $R>0$
\begin{equation}\label{radial1}
	\frac{1}{R^{2s_c}}\int_{|x|\leq R}|u|^2\,dx\leq c\|u\|_{L^{\sigma_c}}^2
\end{equation}
and
\begin{equation}\label{radial2}
	\lim_{R\to+\infty}\frac{1}{R^{2s_c}}\int_{|x|\leq R}|u|^2\,dx=0
\end{equation}
(see Merle and Raph\"ael in \citep[Lemma 1]{MR_Bsc} and also \cite[Lemma 2.1]{CF20}). We should point out that no radial symmetry is needed to obtain these estimates. Next, we prove a crucial interpolation inequality for general functions in $\dot{H}^{s_c}\cap\dot{H}^{1}$.

\begin{lemma}\label{lemaradialGN}[Non-radial Gagliardo-Nirenberg inequality] Suppose that $N\geq 3$, $0<b<2$, $\frac{2-b}{N}<\sigma<\min\{\frac{2}{N},\frac{2-b}{N-2}\}$.
		Then, for all $\eta>0$, there exists a constant $C_\eta>0$ such that for all $u\in \dot{H}^{s_c}\cap\dot{H}^{1}$ the following inequality holds
		\begin{equation}\label{GNradial}
		\int_{|x|\geq R}|x|^{-b}|u|^{2\sigma+2}\,dx\leq \eta\|\nabla u\|_{L^{2}}^{2}+\frac{C_{\eta}}{R^{2(1-s_ c)}}[\rho(u,R)]^{\frac{2\sigma+2-\sigma N}{2-\sigma N}}.
		\end{equation}
\end{lemma}
\begin{proof}
Since $\rho(u,R)$ is non-increasing in $R>0$, given $j\in \mathbb{N}$, we first note that $\rho(u,2^jR)\leq \rho(u,R)$.
Now, for each $j\in \mathbb N$, let
$$
\mathcal C_j=\{x\in \mathbb R;\,2^jR\leq |x|\leq 2^{j+1}R\}.
$$
By interpolation and Sobolev embedding, we first have\footnote{As in the proof of Theorems \ref{thmblowdisper} and \ref{Tlpc}, this is the step where we use the decaying factor in the nonlinearity to replace the radial assumption.}
\begin{align}
	\int_{C_j}|x|^{-b}|u|^{2\sigma+2}\,dx&\leq c\frac{1}{(2^jR)^b}\|u\|_{L^{\frac{2N}{N-2}}(\mathcal C_j)}^{\sigma N}\|u\|_{L^2(\mathcal C_j)}^{2\sigma+2-\sigma N}\\
	&\leq c\|\nabla u\|_{L^2}^{\sigma N}\frac{1}{(2^jR)^{(1-s_c)(2-\sigma N)}}[\rho(u,2^jR)]^{\sigma+1-\frac{\sigma N}{2}},\label{holder}
\end{align}
where in the last inequality we have also used the definition \eqref{defrho} and the fact that
$$
b-s_c(2\sigma +2-\sigma N)=(1-s_c)(2-\sigma N)>0.
$$
Let ${\theta}\in (0,(1-s_c)(2-\sigma N))$. Given $\tilde \eta>0$, by Young's inequality, there exists a constant $\tilde C_{ \tilde \eta}>0$ such that
\begin{align}
	\int_{C_j}|x|^{-b}|u|^{2\sigma+2}\,dx\leq  \frac{\tilde \eta}{(2^j)^{\frac{2\theta}{\sigma N}}} \|\nabla u\|_{L^2}^2+\frac{\tilde C_{\tilde\eta}}{(2^jR)^{2(1-s_c)}(2^j)^{-\frac{2\theta}{2-\sigma N}}}[\rho(u,2^j R)]^{\frac{2\sigma+2-\sigma N}{2-\sigma N}}.\label{GNann}
\end{align}

	Therefore, from \eqref{GNann}, we deduce that
	\begin{align}
	\int_{|x|\geq R}|x|^{-b}|u(x)|^{2\sigma+2}\,dx&=\sum_{j=0}^{\infty}\int_{C_j}|x|^{-b}|u|^{2\sigma+2}\,dx\\
	&\leq \tilde \eta\left(\sum_{j=0}^{\infty}\frac{1}{(2^{\frac{2\theta}{\sigma N}})^j}\right)\|\nabla u\|_{L^2}^{2}\\
	&\quad +\frac{C_{\tilde\eta}}{R^{2(1-s_c)}}\left(\sum_{j=0}^{\infty}\frac{1}{\left(2^{^{\frac{2[(1-s_c)(2-\sigma N)-\theta]}{2-\sigma N}}}\right)^j}\right)[\rho(u, R)]^{\frac{2\sigma+2-\sigma N}{2-\sigma  N}},\nonumber
	\end{align}
which completes the proof of the lemma since the above two series are summable.
\end{proof} 

\subsection{Existence of blow-up solutions and lower bound for the blow-up rate}
 
We now prove some uniform estimates that are the main ingredients in the proof of Theorems \ref{scteo1}-\ref{scteo2} and Corollary \ref{cor13}. The technique is very similar to the one used in the proof of \cite[Propositions 4.1-4.2]{CF20} and was inspired by the work of \citet{MR_Bsc}, these papers treat the radial setting. Here, we show that Lemmas \ref{lemintradial}-\ref{lemaradialGN} allow us to consider general solutions of the INLS equation \eqref{PVI}.

\begin{prop}\label{prop1}
Let $N\geq3$, $0<b<\min\{\frac{N}{2},2\}$, $\frac{2-b}{N}<\sigma<\min\left\{\frac{2-b}{N-2},\frac{2}{N}\right\}$ and $v_0 \in  \dot H^{s_c}\cap \dot H^1$ such that the corresponding solution $v\in C\left([0,\tau_\ast]: \dot{H}^{s_c}\cap\dot{H}^1\right)$ to \eqref{PVI}  satisfies
	\begin{equation}\label{hpprop1i}
	\tau_\ast^{1-s_c}\max\{E[v_0],0\}<1
	\end{equation}
	and
	\begin{align}\label{hpprop1ii}
	M_0:=\frac{4\|v_0\|_{L^{\sigma_c}}}{\|V\|_{L^{\sigma_c}}}\geq 2,
	\end{align}
where $V$ is a solution to elliptic equation \eqref{elptcpc1} with minimal $L^{\sigma_c}$-norm. Then, there exist universal constants $C_1,\alpha_1,\alpha_2>0$ depending only on $N,\sigma$ and $b$  such that, for all $\tau_0\in [0,\tau_\ast]$, the following uniform estimates hold
	\begin{align}\label{prop1ii}
	\rho(v(\tau_0),M_0^{\alpha_1}\sqrt{\tau_0})\leq C_1M_0^2
	\end{align}
and 
	\begin{align}\label{prop1i}
	\int_0^{\tau_0}(\tau_{0}-\tau)\|\nabla v(\tau)\|_{L^2}^2\,d\tau \leq M_0^{\alpha_2}\tau_{0}^{1+s_c}.
	\end{align}
\end{prop}

\begin{proof} For all $A>0$ and $\tau_0\in [0,\tau_{\ast}]$, let $R=A\sqrt{\tau_0}$ and $M_{\infty}$ defined by
	\begin{align}\label{Minfty}
	M^2_{\infty}(A,\tau_0)=\max_{\tau\in[0,\tau_0]}\rho(v(\tau),A\sqrt{\tau}).
	\end{align}
Then using the non-radial Gagliardo-Nirenberg inequality \eqref{GNradial} and following the proof of Lemma 4.3 in \cite{CF20} we obtain the existence of a universal constant $c>0$ such that
	\begin{align}
	8\sigma s_c\int_{0}^{\tau_0}(\tau_0-\tau)\|\nabla v(\tau)\|_{L^2}^2\,d\tau&\leq  c\tau_0^{1+s_c}\left[A^{2(1+s_c)}\|v_0\|_{L^{\sigma_c}}^2+\frac{[M^2_{\infty}(A,\tau_0)]^{\frac{2\sigma+2-\sigma N}{2-\sigma N}}+M^2_{\infty}(A,\tau_0)}{A^{2(1-s_c)}}\right]\nonumber\\
	&\quad\quad+2\tau_0\left[z'_R(0)+8\tau_0(\sigma s_c+1) E[v_0]\right] \label{lemi}
	\end{align}
	and
	\begin{align}
	\frac{1}{R^{2s_c}}\int_{R\leq |x|\leq 2R}|v(\tau_0)|^2\,dx&\leq c\|v_0\|_{L^{\sigma_c}}^2+\frac{c}{A^4}\left[[M^2_{\infty}(A,\tau_0)]^{\frac{2\sigma+2-\sigma N}{2-\sigma N}}+M^2_{\infty}(A,\tau_0)\right]\\
	&\quad\quad+\frac{2}{\tau_0^{s_c}A^{2(1+s_c)}}\left[z'_R(0)+8\tau_0(\sigma s_c+1)  E[v_0]\right]. \,\,\,\,\,\, \label{lemii}
	\end{align}

Now, let $\varepsilon>0$ a fixed small enough real number to be chosen later and define
\begin{align}\label{Gep}
G_{\varepsilon}=M_0^{\frac{1}{\varepsilon}}\,\,\, \mbox{and}\,\,\,A_{\varepsilon}=\left(\frac{\varepsilon G_{\varepsilon}}{M_0^2}\right)^{\frac{1}{2(1+s_c)}},
\end{align}
where $M_0$ is given in \eqref{hpprop1ii}. Consider the following estimates
\begin{align}\label{dispersioneps}
\int_{0}^{\tau_0}(\tau_0-\tau)\|\nabla v(\tau)\|_{L^2}^{2}\,d\tau\leq G_{\varepsilon}\tau_0^{1+s_c}
\end{align}
and 
\begin{align}\label{Minftyep}
M_\infty^2(A_\varepsilon,\tau_0)\leq \frac{2M_0^2}{\varepsilon}.
\end{align}
We define 
\begin{align}\label{tau11}
S_{\varepsilon}= \left\{\tau\in [0,\tau_*];\,\,\eqref{dispersioneps}\mbox{ and }\eqref{Minftyep} \mbox{ hold for all }\tau_0\in [0,\tau]\right\}
\end{align}
and 
\begin{align}\label{tau1}
\tau_1=\max_{\tau\in [0,\tau_*]}S_{\varepsilon}.
\end{align}
Note that $S_{\varepsilon}\neq \emptyset$, for $\varepsilon>0$ sufficiently small, since $v\in C\left([0,\tau_\ast]: \dot{H}^{s_c}\cap\dot{H}^1\right)$ and also by the definition of $M_0$ in \eqref{hpprop1ii} and the Sobolev embedding $\dot{H}^{s_c} \subset L^{\sigma_c}$. 

The goal is to show that $\tau_1=\tau_\ast$ and therefore inequalities \eqref{dispersioneps} and \eqref{Minftyep} hold at the maximal time $\tau_{\ast}$, which clearly imply \eqref{prop1ii} and \eqref{prop1i}, in view of definition \eqref{Gep}. 	Indeed, from \eqref{Gep} and \eqref{Minftyep} it is easy to see that
	\begin{align}\label{110}
	\frac{[M^2_{\infty}(A_{\varepsilon},\tau_0)]^{\frac{2\sigma+2-\sigma N}{2-\sigma N}}+M^2_{\infty}(A_{\varepsilon},\tau_0)}{A_{\varepsilon}^{2(1-s_c)}}
&\leq \frac{\left(\frac{2M_0^2}{\varepsilon}\right)^{\frac{2\sigma+2-\sigma N}{2-\sigma N}}+\frac{2M_0^2}{\varepsilon}}{\left(\frac{\varepsilon G_{\varepsilon}}{M_0^2}\right)^{\frac{1-s_c}{1+s_c}}}\leq c\frac{1}{M_0^{\frac{1}{\varepsilon}\cdot \frac{1-s_c}{1+s_c}}} \left(\frac{M_0^2}{\varepsilon}\right)^{\frac{2\sigma+2-\sigma N}{2-\sigma N}+\frac{1-s_c}{1+s_c}}\leq \frac{1}{10}, 
	\end{align}
for $\varepsilon>0$ small enough, where we have used the assumption \eqref{hpprop1ii}. Combining the last inequality with \eqref{lemi} and using again definition \eqref{Gep} and assumption \eqref{hpprop1ii}, we have, for $R=A_{\varepsilon}\sqrt{\tau_0}$, that
	\begin{align}
	\int_{0}^{\tau_0}(\tau_0-\tau)\|\nabla v(\tau)\|_{L^2}^2\,d\tau&\leq  c\tau_0^{1+s_c}\left[M_0^2A_{\varepsilon}^{2(1+s_c)}+\frac{[M^2_{\infty}(A_{\varepsilon},\tau_0)]^{\frac{2\sigma+2-\sigma N}{2-\sigma N}}+M^2_{\infty}(A_{\varepsilon},\tau_0)}{A_{\varepsilon}^{2(1-s_c)}}\right]\nonumber\\
	&\quad+2c\tau_0\left[z'_R(0)+8\tau_0(\sigma s_c+1) E[v_0]\right]\\
	& \leq c\tau_0^{1+s_c}\left[\varepsilon G_{\varepsilon}+\frac{1}{10}\right]+2c\tau_0\left[z'_R(0)+8\tau_0(\sigma s_c+1) E[v_0]\right]\nonumber\\
	& \leq G_{\varepsilon}\tau_0^{1+s_c}\left\{c\varepsilon+\frac{c}{10G_{\varepsilon}}+\frac{2c}{G_{\varepsilon}\tau_0^{s_c}}\left[z'_R(0)+8\tau_0(\sigma s_c+1)E[v_0]\right]\right\}\nonumber\\
	&\leq G_{\varepsilon}\tau_0^{1+s_c}\left\{\frac{1}{10}+\frac{2c}{G_{\varepsilon}\tau_0^{s_c}}\left[z'_R(0)+8\tau_0(\sigma s_c+1)E[v_0]\right]\right\},\nonumber
	\end{align}
for $\varepsilon>0$ small enough. 

Next, from Lemma 4.4 in \cite{CF20}\footnote{In the proof of \cite[Lemma 4.4]{CF20} does not require a radial assumption.}, there exists $\varepsilon_0>0 $ small enough and $c>0$ a universal constant such that, for all $\tau_0\in[0,\tau_1]$, $0<\varepsilon\leq \varepsilon_0$, $A\geq A_\varepsilon$ and $R=A\sqrt{\tau_0}$, the following inequality holds
	\begin{align}\label{ImE}
	z'_R(0)+8\tau_0(\sigma s_c+1)E[v_0]\leq c \frac{M_0^2A^{2(1+s_c)}}{\varepsilon^{\frac{1}{1+s_c}}}\tau_0^{s_c}.
	\end{align}
Therefore, from estimate \eqref{ImE} with $A=A_\varepsilon$, for all $\tau_0\in [0,\tau_1]$ and since $R=A_{\varepsilon}\sqrt{\tau_0}$, we get
	\begin{align}\label{dispersioneps2}
	\int_{0}^{\tau_0}(\tau_0-\tau)\|\nabla v(\tau)\|^2_{L^2}\,d\tau&\leq G_{\varepsilon}\tau_0^{1+s_c}\left[\frac{1}{10}+c\frac{M_0^2A_{\varepsilon}^{2(1+s_c)}}{G_{\varepsilon}\varepsilon^{\frac{1}{1+s_c}}}\right]\\
	&=G_{\varepsilon}\tau_0^{1+s_c}\left[\frac{1}{10}+c\varepsilon^{\frac{s_c}{1+s_c}}\right]\leq \frac{G_{\varepsilon}}{2}\tau_0^{1+s_c}.
	\end{align}
Moreover, let $A\geq A_{\varepsilon}$ and $R=A\sqrt{\tau_0}$. First, since $\rho$ is non-increasing in $R$, then for all $\tau_0\in [0,\tau_1]$ we obtain
	\begin{align}
	M^2_{\infty}(A,\tau_0)\leq M^2_{\infty}(A_{\varepsilon},\tau_0)\leq M^2_{\infty}(A_{\varepsilon},\tau_1)\leq \frac{2M_0^2}{\varepsilon},
	\end{align}
where in the last inequality we have used \eqref{Minftyep}.\\
Combining \eqref{lemii} with \eqref{hpprop1ii}, the last inequality and \eqref{ImE} we deduce
	\begin{align}\label{Minftyep2}
	\frac{1}{R^{2s_c}}\int_{R\leq |x|\leq 2R}|v(\tau_0)|^2\,dx&\leq cM^2_0+\frac{c}{A^4}\left[[M^2_{\infty}(A,\tau_0)]^{\frac{2\sigma+2-\sigma N}{2-\sigma N}}+M^2_{\infty}(A,\tau_0)\right]\\
	&\quad +\frac{2}{\tau_0^{s_c}A^{2(1+s_c)}}\left[z'_R(0)+8\tau_0(\sigma s_c+1) E[v_0]\right]\nonumber\\
	&\leq cM_0^2+c\frac{\left(\frac{2M_0^2}{\varepsilon}\right)^{\frac{2\sigma+2-\sigma N}{2-\sigma N}}+\frac{2M_0^2}{\varepsilon}}{\left(\frac{\varepsilon G_{\varepsilon}}{M_0^2}\right)^{\frac{2}{1+s_c}}}+\frac{cM_0^2}{\varepsilon^{\frac{1}{1+s_c}}}\\
	&=\frac{M_0^2}{\varepsilon}\left[c\varepsilon+\frac{c\varepsilon}{M_0^{2+\frac{1}{\varepsilon}\cdot \frac{2}{1+s_c}}} \left(\frac{M_0^2}{\varepsilon}\right)^{\frac{2\sigma+2-\sigma N}{2-\sigma N}+\frac{2}{1+s_c}}+c\varepsilon^{\frac{s_c}{1+s_c}}\right]< \frac{M_0^2}{\varepsilon},
	\end{align}
for $\varepsilon>0$ small enough. 
	
Finally, in view of the regularity of $v$ and the definition of $S_{\varepsilon}$ in \eqref{tau11}, the inequalities \eqref{dispersioneps2} and  \eqref{Minftyep2} imply that $\tau_1=\tau_\ast$ and the proof of Proposition \ref{prop1} is completed.
\end{proof}

Next, we prove a lower bound on the $L^2$ norm of the initial data around the origin, assuming an additional restriction on the energy.
\begin{prop}\label{prop2}
Let $0<b<2$, $\frac{2-b}{N}<\sigma<\min\left\{\frac{2-b}{N-2},\frac{2}{N}\right\}$ and $v\in
 C\left([0,\tau_*]: \dot H^{s_c}\cap \dot H^1\right)$ a solution to \eqref{PVI} with initial data $v_0 \in  \dot H^{s_c}\cap \dot H^1$ such that \eqref{hpprop1i} and \eqref{hpprop1ii} of Proposition \ref{prop1} hold. Let 
	\begin{align}\label{hpprop2i}
	\tau_0\in\left[0,\frac{\tau_*}{2}\right].
	\end{align}
Define $\lambda_v(\tau)=\|\nabla v(\tau)\|^{-\frac{1}{1-s_c}}_{L^2}$ and assume that
	\begin{align}\label{Ev}
	E[v_0]\leq \frac{\|\nabla v(\tau_0)\|^2_{L^2}}{4}=\frac{1}{4\lambda_v^{2(1-s_c)}(\tau_0)}.
	\end{align}
Then, there exist universal  constants $C_2, \alpha_3>0$ depending only on $N,\sigma$ and $b$ such that if
	\begin{align}\label{F}
	F_*=\frac{\sqrt{\tau_0}}{\lambda_v(\tau_0)}\,\,\,\mbox{ and }\,\,\,D_*=M_0^{\alpha_3}\max[1,F_*^{\frac{1+s_c}{1-s_c}}],
	\end{align}
	then
	\begin{align}\label{tprop2}
	\frac{1}{\lambda_v^{2s_c}(\tau_0)}\int_{|x|\leq D_*\lambda_v(\tau_0)}|v_0|^2\,dx\geq C_2.
	\end{align}
\end{prop}
\begin{proof}
	The argument is quite similar to the one in \cite[Proposition 4.2]{CF20} and we only sketch the details. Given $v(\tau_0)$ we define
	\begin{align}
	w(x)=\lambda_v^{\frac{2-b}{2\sigma}}(\tau_0)v(\lambda_v(\tau_0)x,\tau_0),
	\end{align}
where $\lambda_v(\tau_0)=\|\nabla v(\tau_0)\|^{-\frac{1}{1-s_c}}_{L^2}$. By a straightforward calculation and assumption \eqref{Ev}, we have
\begin{equation}\label{wgrad}
\|\nabla w\|_{L^2}^2=1\,\,\, \mbox{and} \,\,\, E[w]\leq \frac{1}{4}.
\end{equation}
Thus, by the definition of the energy \eqref{Energy}, we get
\begin{equation}\label{xbw}
\int |x|^{-b}|w|^{2\sigma+2}\,dx\geq \frac{\sigma+1}{2}.
\end{equation}
Recalling the definition of $F_{\ast}$ in \eqref{F}, for $J\geq M_0^{\alpha_1}F_*$, the estimate \eqref{prop1ii} obtained in Proposition \ref{prop1} and the definition of the semi-norm $\rho$ \eqref{defrho} imply
\begin{equation}\label{wJast}
\rho(w,J)\leq \rho(v(\tau_0),M_0^{\alpha_1}\sqrt{\tau_0})\leq c M_0^2.
\end{equation}
On the other hand, for $J\geq CM_0^{\frac{2\sigma+2-\sigma N}{(1-s_c)(2-\sigma N)}}$ and all $\eta>0$, from estimate \eqref{GNradial} we get
	\begin{align}
	\int_{|x|\geq J}|x|^{-b}|w|^{2\sigma+2}\,dx&\leq \eta \|\nabla w\|_{L^2}^2+\frac{C_\eta}{J^{2(1-s_c)}}[\rho(w,J)]^{\frac{2\sigma+2-\sigma N}{2-\sigma N}}\\
	&\leq \eta+\frac{C_\eta}{C^{2(1-s_c)}M_0^{\frac{2(2\sigma+2-\sigma N)}{2-\sigma N}}}[\rho(w,J)]^{\frac{2\sigma+2-\sigma N}{2-\sigma N}}.
	\end{align}
Now, taking
\begin{align}\label{Aep}
	J_{*}=C_{*}\max[M_0^{\alpha_1}F_*,M_0^{\frac{2\sigma+2-\sigma N}{(1-s_c)(2-\sigma N)}}],
	\end{align}
for $C_*>1$ large enough, the previous two inequalities (with $\eta>0$ small enough) and \eqref{xbw} yield
	\begin{align}\label{sigma+14}
	\int_{|x|\leq J_*}|x|^{-b}|w|^{2\sigma+2}\,dx\geq\frac{\sigma+1}{4}.
	\end{align}

Next, we apply the Gagliardo-Nirenberg type inequality 
\begin{align}\label{GNF}
	\int|x|^{-b}|f|^{2\sigma+2}\,dx\leq c\|\nabla f\|_{L^2}^{2\sigma s_c+2}\|f\|_{L^2}^{2\sigma(1-s_c)}
	\end{align}
(see \cite[Theorem 1.2]{Farah}) to deduce a lower bound on the $L^2$ norm of $v(\tau_0)$ around the origin. Indeed, let $\varphi\in C^{\infty}_0(\Real^N)$ be a cut-off function such that
	\begin{align}\label{fi}
	\varphi(x)=\left\{
	\begin{array}{ll}
	1,&|x|\leq 1\\
	0,&|x|\geq 2
	\end{array}
	\right.
	\,\,\,\mbox{ and }\,\,\, \varphi_{A}(x)=\varphi\left(\frac{x}{{A}}\right).
	\end{align}
It is easy to see that
	\begin{align}
	\|\nabla (\varphi_{J_*} w)\|_{L^2}&=\|\nabla \varphi_{J_*} w+\varphi_{J_*}\nabla w\|_{L^2}\leq \|\nabla \varphi_{J_*} w\|_{L^2}+\|\varphi_{J_*}\nabla w\|_{L^2}\\&\leq c(\|w\|_{L^2(|x|\leq 2J_*)}+1),
	\end{align}
where we have used \eqref{wgrad} in the last inequality.  Combining  \eqref{sigma+14}, \eqref{GNF} and the previous inequality, we deduce
	\begin{align}
	\frac{\sigma+1}{4}&\leq \int_{|x|\leq J_*} |x|^{-b}|w|^{2\sigma+2}\,dx\leq \int|x|^{-b}|\varphi_{J_*} w|^{2\sigma+2}\,dx\nonumber\\&\leq c\|\nabla (\varphi_{J_*} w)\|_{L^2}^{2\sigma s_c+2}\|\varphi_{J_*} w\|_{L^2}^{2\sigma(1-s_c)}\leq c\left(\|w\|_{L^2(|x|\leq 2J_*)}^{2\sigma+2}+\|w\|_{L^2(|x|\leq 2J_*)}^{2\sigma(1-s_c)}\right).
	\end{align}
	Therefore, there exists a constant $c_3>0$ such that
	\begin{align}\label{364}
	c_3\leq \int_{|y|\leq 2 J_*}|w|^2\,dx=\frac{1}{\lambda_v^{2s_c}(\tau_0)}\int_{|x|\leq 2J_*\lambda_v(\tau_0)}|v(\tau_0)|^2\,dx.
	\end{align}
	
Finally, the end of the proof of \cite[Proposition 4.2]{CF20} implies that the above estimate is also verified at the initial time, where we use the assumption \eqref{hpprop2i},  and complete the proof of Proposition \ref{prop2}.

\end{proof}

Once Propositions \ref{prop1} and \ref{prop2} are established, the proofs of Theorem \ref{scteo1}-\ref{scteo2} and Corollary \ref{cor13} follow from exactly the same arguments as those of in \cite[Theorem 1.1-1.2 and Corollary 1.3]{CF20}  and hence need not be repeated here.

\vspace{0.5cm}
\noindent 
\textbf{Acknowledgments.} L.G.F. was partially supported by Coordena\c{c}\~ao de Aperfei\c{c}oamento de Pessoal de N\'ivel Superior - CAPES, Conselho Nacional de Desenvolvimento Cient\'ifico e Tecnol\'ogico - CNPq and Funda\c{c}\~ao de Amparo a Pesquisa do Estado de Minas Gerais - Fapemig/Brazil. 

\newcommand{\Addresses}{{
		\bigskip
		\footnotesize
		
		MYKAEL A. CARDOSO, \textsc{Department of Mathematics, UFPI, Brazil}\par\nopagebreak
		\textit{E-mail address:} \texttt{mykael@ufpi.edu.br}
		
		\medskip
		
		LUIZ G. FARAH, \textsc{Department of Mathematics, UFMG, Brazil}\par\nopagebreak
		\textit{E-mail address:} \texttt{farah@mat.ufmg.br}

}}
\setlength{\parskip}{0pt}
\Addresses

\end{document}